\documentclass[10pt,leqno]{amsart}

\usepackage{indentfirst,csquotes}

\usepackage[english]{babel}
\usepackage{amsmath}
\usepackage{amssymb}
\usepackage{bbold}
\usepackage[parfill]{parskip}
\usepackage{amsthm}
\usepackage{csquotes}
\usepackage{enumitem}
\usepackage{graphicx}
\usepackage{xcolor}
\usepackage[bookmarks,hyperfootnotes=false, psdextra=true]{hyperref}
\usepackage[capitalise,nameinlink,noabbrev]{cleveref}
\hypersetup{
    draft = false,
    bookmarksopen=true,
    colorlinks,
    linkcolor={red!60!black},
    citecolor={green!60!black},
    urlcolor={blue!60!black}
}

\newtheorem{theorem}{Theorem}
\newtheorem{prop}[theorem]{Proposition}
\newtheorem{cor}[theorem]{Corollary}
\newtheorem{lemma}[theorem]{Lemma}
\theoremstyle{definition}

\newtheorem{remark}[theorem]{Remark}
\newtheorem{example}[theorem]{Example}
\newcommand{\ora}{\overrightarrow}

\DeclareMathOperator*{\argmax}{arg\,max}
\DeclareMathOperator*{\argmin}{arg\,min}

\begin{document}

\title{Witnessing and Guiding Sets of Tangles}
\author[A Seibt]{Annegret Seibt}
\address{Universität Hamburg, Germany}
\email{annegret.seibt@kuleuven.be}
\date{}
\maketitle
\begingroup
\renewcommand\thefootnote{}\footnotetext{This is the full version of a manuscript under review.}
\addtocounter{footnote}{-1}
\endgroup

\begin{abstract}
Tangles offer a way to indirectly but precisely capture cluster-like though possibly fuzzy substructures in discrete data.

In this paper, we analyze witnessing and guiding sets of tangles that can help to find proper cluster candidates for given tangles.

We show that every $k$-tangle has a witnessing set whose size is bounded in an exponential function in $k$ which improves a result of Grohe and Schweizer~\cite{grohe2015}.

Further, we generalize a result of Diestel, Elbracht and Jacobs~\cite{diestel2021} by providing a characterization of tangles that have a guiding function of some given reliability.

\end{abstract}

\section{Introduction}

\subsection{Tangles}
Tangles are a relatively novel tool for capturing cluster-like though possibly fuzzy substructures in discrete data in an indirect but precise way.
The main idea of tangles is completely different from traditional clustering approaches:
Tangles are not sets of objects with similar features, but sets of features that occur together consistently.
A set of features will be called `consistent' if, for every three features in it, there is at least one object that has these three features.
This property is weak enough for the discovery of fuzzy clusters (with objects that are not pairwise similar).
But at the same time it is strong enough to prove strong structural theorems about tangles~\cite{robertsonseymour}.

The origin of tangles is in the twenty papers spanning proof of Robertson and Seymour's graph minor theorem~\cite{robertsonseymour} where tangles play a central role in describing the global structure of a graph~\cite{robertson2004}.
They have since been generalized from their initial definition in the context of graphs to a purely algebraic concept as tangles of `abstract separation systems'~\cite{diestel2018}.
This opened up a much wider scope of application for tangles while at the same time preserving the validity of the two main theorems of tangle theory, the tree-of-tangles theorem~\cite{diestel2019profiles} and the tangle-tree duality theorem~\cite{diestel2021duality}.
Recently, Diestel~\cite{diestel2024} proposed the use of tangles for data clustering problems and indicated applications in various areas including natural sciences (e.g. clustering illnesses, target-driven drug development), social sciences (e.g. discovering mindsets, identifying meaning) and data sciences (e.g. decoding images and texts)~\cite{diestel2024}.
The practical usefulness and flexibility of tangles for such clustering applications has been demonstrated by Klepper et al.\ \cite{klepper2023} who elaborated an algorithmic framework that translates tangles into a machine learning context.

As Klepper et al.~\cite{klepper2023}, we consider tangles of bipartitions.
Suppose we are given a finite groundset $V$ of points (corresponding to the objects we want to cluster) and a set $S$ of bipartitions $\{A,A^c\}$ of $V$ that we refer to as \emph{separations}.
For each separation $\{A,A^c\}$, there are two \emph{oriented separations}, $A$ and $A^c$, that we can imagine as corresponding to two inverse features shared by all the points in $A$ and $A^c$, respectively.
A \emph{tangle} of $S$ is a set $\tau$ that contains precisely one of $A$ and $A^c$ for every separation $\{A,A^c\} \in S$ and satisfies the following \emph{consistency condition}:\footnote{For an explanation why we consider any three oriented separations here and not only two or even four, see \cite[ch.\ 7.3]{diestel2024}.}
\begin{equation}
\label{conscond}
\forall \{A_1, A_2, A_3\} \subset \tau: \quad A_1 \cap A_2 \cap A_3 \ne \emptyset.
\end{equation}
Note that a tangle is a set of oriented separations. So we can imagine a tangle as representing a set of features that describe a type of object.
This type can be rather idealistic since the consistency condition does not imply that there is a point in $V$ having all the features in $\tau$. It does not even imply that there is a point in $V$ that satisfies most of the features in $\tau$. In fact, it only says that for any three features in $\tau$ there is one point in $V$ that has these three features.
This makes tangles especially suitable for capturing cluster-like sets of possibly pairwise quite different objects that still all belong to a certain type.
Let $C \subset V$ be such a set. Then maybe we do not find a set of features such that all objects in $C$ share all those features or at least most of them.
But we might find a set of features such that each of them is satisfied by most of $C$, say by more than two thirds of it. Then the set of all the oriented separations corresponding to these features is a tangle $\tau$: Let $\{A_1, A_2, A_3\} \subset \tau$. Then $|A_1^c \cup A_2^c \cup A_3^c| < 3 \cdot (1 -\frac 2 3)|C| = |C|$ and, thus, $A_1 \cap A_2 \cap A_3 \ne \emptyset$.

Since clusters seem to give rise to tangles, we can use tangles to discover clusters in given data.
But this raises the question of what set to choose as the belonging cluster to a given tangle $\tau$ if there is a proper one at all.
In this paper, we contribute to a more structural than algorithmic answer to this question (for the algorithmic one, see~\cite{klepper2023}).
The oriented separations constituting $\tau$ give hints of where in $V$ the cluster is lying.
But for instance, $\bigcap \tau$ is not an appropriate candidate for a cluster as for all `interesting' tangles this intersection will be empty: For every $v \in V$ the \emph{principal tangle} $\tau_v$ is the set of all oriented separations that contain $v$. Then we have $v \in \bigcap \tau_v$. These principal tangles always exist and, thus, do not really carry relevant information.
In order to find a more suitable candidate, we consider two tangle related sets that -- in combination -- could finally make up such a candidate.

\subsection{Witnessing sets of tangles}

For every triple $\{A_1, A_2, A_3\}$ of oriented separations, we say that a point $v \in A_1 \cap A_2 \cap A_3$ \emph{witnesses} the triple $\{A_1, A_2, A_3\}$.
Suppose we are given a tangle $\tau$ of $S$. Then a set $W \subset V$ \emph{witnesses} $\tau$ if for all $\{A_1, A_2, A_3\} \subset \tau$ there is a point in $W$ that witnesses $\{A_1, A_2, A_3\}$.
The consistency condition \eqref{conscond} ensures that the groundset $V$ is such a \emph{witnessing set} of every tangle $\tau$.
Thus, our aim is to show the existence of witnessing sets of bounded size.

The tangles and hence the witnessing sets that we find strongly depend on what set of separations $S$ of $V$ we choose. So we can only hope to bound the size of a witnessing set with respect to the set $S$.
For this, we can use that usually one \emph{orders} the separations in $S$ by assigning non-negative integers to them.
When thinking of the oriented separations as associated to features, we basically want to assign low orders to oriented separations that correspond to fundamental features and high orders to oriented separations that correspond to very specific or no features.
There are many ways how this can be done in various particular contexts \cite[ch.\ 9]{diestel2024}.
In order to make the two main tangle theorems applicable, order functions are often required to be \emph{submodular} (see \cref{prelimsubmod}).
A tangle that orients all the separations of order less than $k$ with respect to a submodular order function is called a $k$-\emph{tangle} .

Grohe and Schweitzer~\cite[Lemma 3.1]{grohe2015} showed that every $k$-tangle has a witnes\-sing set of size bounded in terms of the order $k$. The upper bound $\theta(3k-2)$ they found for the size of a minimal witnessing set is recursively defined by $\theta(0) := 0$ and $\theta(i + 1) := \theta(i)+3^{\theta(i)}$ for $i \in \mathbb N_0$.
We improve this bound to an exponential one.

\begin{theorem}
\label{intch2}
Let $S$ be a set of separations of any set $V$ and let $\tau$ be a $k$-tangle of $S$ regarding to some submodular order function $f$. Then $\tau$ has a witnessing set of size at most $\frac {3^k - 1} 2$.
\end{theorem}

We provide two completely different proofs in \cref{uppb1} and \cref{uppb2}.
The first one follows the same induction approach as the proof by Grohe and Schweitzer and yields the slightly weaker upper bound $\tfrac 1 2 (3^{3k-2}-3^k) + k$ for the size of a minimal witnessing set.

In \cref{lowb}, we give an example of a k-tangle that has no witnessing set of size less than $\binom k 3$.
This shows that any function in k that bounds, for all k-tangles, the size of a minimal witnessing set from above cannot be smaller than $\binom k 3$.

\subsection{Guiding sets of tangles}
A witnessing set witnesses that a tangle $\tau$ is indeed a tangle, i.e.\ it shows that $\tau$ satisfies \eqref{conscond}.
Complementary to this, the second tangle-related set we consider defines its tangle in that it decides by majority vote which orientation of a separation lies in $\tau$.
Formally, we say that a set $G \subset V$ \emph{guides} a tangle $\tau$ if, for every oriented separation $A \in \tau$, $|A \cap G| > \frac 1 2 |G|$.
If, for some $\rho \in [0, 1]$, $|A \cap G| > \rho |G|$, for every $A \in \tau$, then $G$ has \emph{reliability} $\rho$.
So a guiding set is a subset of $V$ with reliability $\rho > \tfrac 1 2$.
Contrary to witnessing sets, not every tangle has a guiding set~\cite{diestel2021}.
So our aim concerning guiding sets is to find out under what circumstances a tangle has one, and if so we aim for maximum reliability.

We approach this question by starting from \emph{guiding functions}:
We say that a function $g : V \rightarrow \mathbb R_+$ with $\sum_{v \in V} g(v) = 1$ \emph{guides} a tangle $\tau$ if its \emph{reliability} satisfies
\begin{equation*}
\rho_g := \min \left\{ \sum_{v \in A} g(v): A \in \tau \right\} > \tfrac 1 2.
\end{equation*}
This is inspired by~\cite{elbracht2018} where Elbracht, Kneip and Tegen showed that such a guiding function exists for every graph tangle.
But they also provided an example showing that tangles on bipartitions need not always have a guiding function.
In \cref{guidfct}, we generalize a result of Diestel, Elbracht and Jacobs \cite[Thm.~8]{diestel2021} about the existence of guiding functions for tangles (of bipartitions) with a given reliability, and explain how the maximum reliability of a guiding function can be determined.

If a set $G$ guides a tangle $\tau$ then $g : V \rightarrow \mathbb R_+$ given by $g(v) := \frac 1 {|G|} \mathbb 1_{\{v \in G\}}$ is a guiding function of $\tau$. So guiding functions are a generalization of guiding sets.
Thus, the non-existence of a guiding function of a tangle $\tau$ also implies the non-existence of a guiding set of $\tau$.
In \cref{existguid}, we use a probabilistic argument to show that, given some technical properties, the existence of a guiding function with sufficiently high reliability implies the existence of a guiding set.

\subsection{Interplay of witnessing and guiding sets of tangles}

After dealing with witnessing and guiding sets separately, we want to finally pursue our goal of combining them and finding proper clusters associated to tangles.
Diestel suggests the following sets as the tangle induced clusters of $V$: all subsets of $V$ that guide a non-principal and non-fake\footnote{Requiring \emph{non-fakeness} rules out so-called \emph{black hole tangles} that do not spring from a cluster (see \cref{prelimsubmod}).} tangle with maximum reliability~\cite[ch.\ 14.2]{diestel2024}.
Moreover, he contemplates to ask that the sets should witness their tangle, too.
In \cref{interplay}, we give a counterexample that shows that guiding a non-principal and non-fake tangle with maximum reliability does not imply witnessing this tangle.
So requiring clusters to be witnessing their tangles as well could be a reasonable addition.

Note that according to Diestel's suggestion, what defines a cluster are not properties of its single points but properties of the whole set.
This leads to \emph{fuzzy clusters}: clusters that are (other than soft clusters) concrete subsets of $V$ but (unlike hard clusters) need not be disjoint.
Notice further that such clusters do not have to exist for every tangle.
As mentioned before, there are clusters without guiding sets~\cite{diestel2021}.
Besides in \cref{interplay}, we provide a counterexample showing that having a witnessing set (which is smaller than $V$) and a guiding set does not imply that a tangle has a set that is both witnessing and guiding.

\section{Preliminaries}
\label{prelim}

In this section, we recall the important concepts, definitions and notations that we will use in this paper.
They are mainly taken from \cite{diestel2024} and slightly simplified as we are only dealing with the special case of tangles on bipartitions.

\subsection{Tangles}
\label{prelimtangles}
Suppose we are given a finite groundset $V$ of points and a set $S$ of bipartitions $\{A,A^c\}$ of $V$.
Throughout this paper, we always assume $V$ and $S$ to be finite.
Following the more general notion of tangles, we call the bipartitions in $S$ \emph{separations}.
Each separation $\{A,A^c\} \in S$ has two \emph{sides}, $A$ and $A^c$.
By choosing one of its two sides, we can \emph{orient} a separation.
We denote the set of all sides of the separations in $S$ by  
\begin{equation}
\ora S := \bigcup S
\end{equation}
and refer to them as \emph{oriented separations}.\footnote{The set of all sides does not coincide with the set of the oriented separations in the general notion of tangles. This only happens here as $S$ consists of bipartitions of $V$.}
A subset $\tau \subset \ora S$ containing exactly one of the two sides of every separation $\{A,A^c\} \in S$ is a \emph{tangle} of $S$ if it satisfies the aforementioned \emph{consistency condition}:
\begin{equation*}
\tag{\ref{conscond}}
\forall \{A_1, A_2, A_3\} \subset \tau: \quad A_1 \cap A_2 \cap A_3 \ne \emptyset.
\end{equation*}
If for some separation $s = \{A,A^c\} \in S$, we have $A \in \tau$ for some tangle $\tau$, then we call $\tau(s):= A$ the \emph{big side} of $s$ and $A^c$ the \emph{small side} of $s$ (regarding to $\tau$).
So we can rephrase the consistency condition to: The intersections of any three big sides of separations oriented by a tangle must not be empty.

\subsection{Order functions and submodularity}
\label{prelimsubmod}
When talking about maximal or minimal oriented separations, we refer to the subset relation ``$\subset$'' that partially orders the set $\ora S$, i.e.\ for $A,B \subset S$, we use the definition $A \le B \Leftrightarrow A \subset B$. 
Independently of this, we call any function $f:S \rightarrow \mathbb N_0$ an \emph{order function} on $S$.
Consequently, we refer to $f(\{A,A^c\})$ as the \emph{order} of the separation $\{A,A^c\} \in S$ and to $f(A):= f(\{A,A^c\})$ as the order of the oriented separation $A$.
An order function $f:S^* \rightarrow \mathbb N_0$ on the set $S^*$ of all separations of $V$ is \emph{submodular} if for all oriented separations $A,B \in \ora {S^*}$
\begin{equation}
\label{submodcond}
f(A \cap B) + f(A \cup B) \le f(A) + f(B).
\end{equation}
We refer to $(A \cap B)$ and $(A \cup B)$ as the two \emph{corners} of the oriented separations $A$ and $B$.
Suppose we are given a submodular order function $f$.
For $k \ge 1$, we denote the set of all separations in a set $S$ of order less than $k$ by
\begin{equation}
S_k := \{s \in S: f(s) < k\}.
\end{equation}
A tangle of $S_k$ regarding to a submodular order function is called a \emph{tangle of order $k$} or a $k$-\emph{tangle}.

We refer to a set of separations $S$ as \emph{submodular} if for all $A, B \in \ora S$ at least one of the two corners $(A \cap B)$ or $(A \cup B)$ is in $S$ as well.
The submodularity condition for order functions implies that the sets $S^*_k$ with respect to the set $S^*$ of all separations of $V$ are submodular: Let $A, B \in S^*_k$. Then we have $f(A \cap B) + f(A \cup B) \le f(A) + f(B) \le 2k - 2$. But if both $(A \cap B)$ and $(A \cup B)$ were not in $S^*_k$, then $f(A \cap B) + f(A \cup B) \ge 2 k$ -- which is a contradiction.

The sets $S_k$ do not have to be defined with respect to the set $S^*$ of all separations of $V$ but they can be defined for any subset $S \subset S^*$. However, a $k$-tangle (of $S_k$ for some $S \subset S^*$) that cannot be extended to a $k$-tangle of $S^*_k$ regarding to the set $S^*$ of all separations of $V$ is called a \emph{fake} tangle.

\subsection{Witnessing and guiding sets}

We call a set $U \subset V$ a \emph{cover} of a tangle $\tau$ of $S$ if $U \cap A \ne \emptyset$ for all $A \in \tau$.
Similar to this, we refer to a set $W \subset V$ as \emph{witnessing} $\tau$ if
\begin{equation}
\label{witncond}
\forall \{A_1, A_2, A_3\} \subset \tau: \quad W \cap (A_1 \cap A_2 \cap A_3) \ne \emptyset.
\end{equation}
This condition resembles the consistency condition \eqref{conscond} for tangles which shows that the whole groundset $V$ is a witnessing set of every tangle.
We say that the triple $\{A_1, A_2, A_3\} \subset \ora S$ of oriented separations is \emph{witnessed} by a point $v \in V$ if $v \in A_1 \cap A_2 \cap A_3$.
Further, we say that $\{A_1, A_2, A_3\}$ is witnessed by a set $W \subset V$ if there is a point $v \in W$ that witnesses $\{A_1, A_2, A_3\}$.
A tangle $\tau$ with a witnessing set (or cover) of size one is called a \emph{principal tangle}.
A principal tangle $\tau$ is characterized by $\bigcap \tau \ne \emptyset$.
There are always principal tangles since for every $v \in V$ we can construct $\tau_v := \{A \in \ora S: v \in A\}$ with $v \in \bigcap \tau_v$.

Given a tangle $\tau$ on $V$, the \emph{reliability} of a set $G \subset V$ is defined by
\begin{equation}
\label{setreliab}
\rho_G:= \min\left\{\tfrac{|G \cap A|}{|G|}: A \in \tau\right\}.
\end{equation}
If $\rho_G > \tfrac 1 2$ then $G$ is called a \emph{guiding set} of $\tau$.
We can generalize reliability to functions on $V$:
We define the reliability of a function $g: V \rightarrow \mathbb R_+$ with $\sum_{v \in V} g(v) = 1$ by
\begin{equation}
\label{fctreliab}
\rho_g := \min \left\{ \sum_{v \in A} g(v): A \in \tau \right\}.
\end{equation}
Analogously to the definition of guiding sets, we call a function $g: V \rightarrow \mathbb R_+$ with $\sum_{v \in V} g(v) = 1$ a \emph{guiding function} if $\rho_g > \tfrac 1 2$.
Every guiding set $G$ of a tangle $\tau$ induces a guiding function $g$ of $\tau$ with the same reliability as $G$ by $g(v) := \tfrac 1{|G|} \mathbb 1_{v \in G}$.

\section{Witnessing sets of $k$-tangles}
\label{witnsets}
In this section, we give two proofs for an exponential upper bound of the size of a minimal witnessing set.
Further, we provide an example showing that for every order $k \ge 1$ there is a $k$-tangle whose minimal witnessing set has size $\binom k 3$ in \cref{lowb}.

\subsection{First upper bound}
\label{uppb1}

Grohe and Schweitzer~\cite{grohe2015} showed that for every $k$-tangle there is a witnessing set of size bounded in terms of the order $k$.
They proved the upper bound $\theta(3k-2)$ for the size of a  minimal witnessing set that is recursively defined by $\theta(0) := 0$ and $\theta(i + 1) := \theta(i)+3^{\theta(i)}$ for $i \in \mathbb N_0$.
This bound is the result of a proof by induction:
For a given $k$-tangle regarding to a submodular order function $f$, Grohe and Schweitzer construct subsets $W_1 \subset W_2 \subset \dots \subset V$ of the groundset such that the points in $W_i$ witness all triples of oriented separations $\{A_1,A_2,A_3\} \subset \tau$ with $f(A_1) + f(A_2) + f(A_3) < i$.
In this section, we follow the same approach.
But in our induction step, we manage to find bigger sets of triples of oriented separations that are witnessed by the same point $v \in V$.
So in every step, we can get along with adding fewer points to $W_i$ than in the original proof, which will result in an exponential bound.

Grohe and Schweitzer~\cite[ch. 3]{grohe2015} mentioned the following proposition providing an upper bound for the size of a cover in terms of the order of a tangle.

\begin{prop}
\label{coversize}
Let $S$ be a set of separations of any set $V$ and let $\tau$ be a $k$-tangle of $S$ regarding to some submodular order function $f$. Then $\tau$ has a cover of size at most $k$.
\end{prop}
For a proof showing that the upper bound provided in \cref{coversize} is sharp, see \cref{lowb}.
As a corollary, this proposition yields an upper bound for the size of minimal witnessing sets of $k$-tangles that extend to $(3k-2)$-tangles.
\begin{cor}
\label{coversizecor}
Let $S$ be a set of separations of any set $V$ and let $\tau$ be a $k$-tangle of $S$ regarding to some submodular order function $f$. If $\tau$ can be extended to a $(3k-2)$-tangle $\tau'$ of $S$ then $\tau$ has a witnessing set of size $3k-2$.
\end{cor}
\begin{proof}
Let $W$ be a cover of $\tau'$ and let $\{A_1, A_2, A_3\} \in \tau$ be any triple of oriented separations in $\tau$.
Then, by submodularity (i.e.\ using \eqref{submodcond} twice), $f(A_1 \cap A_2 \cap A_3) \le f(A_1) + f(A_2) + f(A_3) < 3k-2$.
Thus, $A_1 \cap A_2 \cap A_3 \in \tau'$.
Therefore, by the definition of a cover $W \cap (A_1 \cap A_2 \cap A_3) \ne \emptyset$.
So, as desired, $W$ is a witnessing set of $\tau$.
\end{proof}
Although the premise of \cref{coversizecor} will be rarely met we can use \cref{coversize} and a version of \cref{coversizecor} in the base case in the proof of the following theorem.

\begin{theorem}
Let $S$ be a set of separations of any set $V$ and let $\tau$ be a $k$-tangle of $S$ regarding to some submodular order function $f$. Then $\tau$ has a witnes\-sing set of size  at most $\tfrac 1 2 (3^{3k-2}-3^k) + k$.
\end{theorem}
\begin{proof}
We construct sets $W_k \subset W_{k+1} \dots \subset V$ inductively such that for $i \ge k$ all $\{A_1,A_2,A_3\} \subset \tau$ with $f(A_1) + f(A_2) + f(A_3) < i$ satisfy $W_i \cap (A_1 \cap A_2 \cap A_3) \ne \emptyset$.
Then $W := W_{3k-2}$ will be a witnessing set of $\tau$.

First, we choose $W_k$ as a minimal cover of $\tau$ whose size, by \cref{coversize}, is $|W_k| \le k$. Then $W_k$ is as desired:
Let $\{A_1, A_2, A_3\} \subset \tau$ with $f(A_1) + f(A_2) + f(A_3) < k$.
By submodularity, we know $f(A_1 \cap A_2 \cap A_3) \le f(A_1) + f(A_2) + f(A_3) < k$.
Thus, $A_1 \cap A_2 \cap A_3 \in \tau$ since $\tau$ is a $k$-tangle.
But then $W_k \cap (A_1 \cap A_2 \cap A_3) \ne \emptyset$ by the definition of a cover.

For the step from $W_k$ to $W_{k+1}$ (and only for this one), we copy the induction step from Grohe and Schweitzer's proof.
Therefore let $\mathcal X_k$ be the set of all partitions of $W_k$ into three (possibly empty) pairwise disjoint sets $\{X_1,X_2,X_3\}$. So $|\mathcal X_k| \le 3^{k-1}$.
Then for every oriented separation triple $\{A_1, A_2, A_3\} \subset \tau$ that is not witnessed by $W_k$, we can define
\begin{equation}
\label{deftripleX}
X_1 := W_k \setminus A_1, \quad X_2 := W_k \setminus (A_2 \cup X_1), \quad X_3 := W_k \setminus (A_3 \cup X_1 \cup X_2).
\end{equation}
These three sets are pairwise disjoint and since $\{A_1, A_2, A_3\}$ is not witnessed by $W_k$,
\begin{equation*}
X_1 \cup X_2 \cup X_3 = (W_k \setminus A_1) \cup (W_k \setminus A_2) \cup (W_k \setminus A_3) = W_k \setminus (A_1 \cap A_2 \cap A_3) = W_k.
\end{equation*}
Hence, $X := \{X_1,X_2,X_3\} \in \mathcal X_k$ and
\begin{equation}
\label{septripleX}
A_j \cap X_j = \emptyset \text{ for } j = 1,2,3.
\end{equation}

Next, we can show that for every $X \in \mathcal X_k$ and for $j = 1,2,3$ there is either no oriented separation $B_{X,j} \in \tau$ with $B_{X,j} \cap X_j = \emptyset$ or there is a unique $B_{X,j} \in \tau$ such that
\begin{enumerate}[label = (\roman*)]
\item $B_{X,j} \cap X_j =\emptyset$;
\item subject to (i), the order $f(B_{X,j})$ is minimal;
\item subject to (i) and (ii), $B_{X,j}$ is minimal (with respect to the subset relation~``$\subset$'') 
\end{enumerate}
This can be proven by a standard submodularity argument:
We consider the case where there is an oriented separation $B \in \tau$ with $B \cap X_j = \emptyset$.
Then the set of all oriented separations satisfying (i) and (ii) is not empty.
Suppose that this set has two different elements $B_1, B_2 \in \tau$ that are both of minimum order $l := f(B_1) = f(B_2)$ and, subject to having order $l$, are both minimal. Since  $B_1$ and $B_2$ are minimal, $|B_1 \cap B_2| > l$.
But then, by submodularity, $|B_1 \cup B_2| \le |B_1| + |B_2| -|B_1 \cap B_2| < l$ -- a contradiction to the minimal order of $B_1$ and $B_2$.
So there is a unique oriented separation $B_{X,j}\in \tau$ satisfying (i) - (iii).

Now we are ready to construct $W_{k+1}$ by adding at most $|\mathcal X_k|$ points to $W_k$:
If for some $X \in \mathcal X_k$, no $B_{X,j}$ exists for any $j \in \{1,2,3\}$, then there is no yet unwitnessed triple $\{A_1, A_2, A_3\} \subset \tau$ satisfying \eqref{septripleX} and nothing has to be done. Otherwise, by consistency condition \eqref{conscond} for a tangle, we can pick a point $v_X \in B_{X,1} \cap B_{X,2} \cap B_{X,3}$. This point $v_X$, by the definition of $B_{X,j}$, witnesses all triples $\{A_1, A_2, A_3\} \subset \tau$ satisfiying \eqref{septripleX} for which $f(A_1) + f(A_2) + f(A_3)$ is minimal. As every triple $\{A_1, A_2, A_3\} \subset \tau$ satisfying \eqref{septripleX} cannot be witnessed by $W_k$ and, thus, has $f(A_1) + f(A_2) + f(A_3) \ge k$, all $\{A_1, A_2, A_3\} \subset \tau$ with $f(A_1) + f(A_2) + f(A_3) \le k$ are witnessed by either $W_k$ or
\begin{equation}
\label{Nk}
N_k := \{v_X : X \in \mathcal X_k\}
\end{equation}
So we can define
\begin{equation}
\label{Wk}
W_{k+1}:= W_k \cup N_k.
\end{equation}
Then $W_{k+1}$ is as desired and $|W_{k+1}| = |W_k| + |N_k| \le |W_k| + |\mathcal X_k| \le k + 3^{k-1}$.

For the induction step, we suppose that $\mathcal X_{i-1}$, $N_{i-1}$, $W_i$ and a function $\mathcal X_{i-1} \rightarrow N_{i-1}$, $X \mapsto v_x$ have been defined for some $i \ge k+1$ such that
\begin{enumerate}[label = (\roman*)]
\item for all triples $\{A_1, A_2, A_3\} \subset \tau$ that are not witnessed by $W_{i-1}$ there is $\{X_1,X_2,X_3\} \in \mathcal X_{i-1}$ with \eqref{septripleX};
\item $v_X \in A_1 \cap A_2 \cap A_3$ for all triples $\{A_1, A_2, A_3\} \subset \tau$ satisfying~\eqref{septripleX} for $X = \{X_1,X_2,X_3\} \in  \mathcal X_{i-1}$ for which $f(A_1) + f(A_2) + f(A_3)$ is minimal. And $|\mathcal X_{i-1}| \le 3^{i-2}$;
\item $f(A_1) + f(A_2) + f(A_3) \ge i-1$ for all triples $\{A_1, A_2, A_3\} \subset \tau$ satisfying \eqref{septripleX} for some $X = \{X_1,X_2,X_3\} \in  \mathcal X_{i-1}$;
\item $W_i \cap (A_1 \cap A_2 \cap A_3) \ne \emptyset$ for all $\{A_1, A_2, A_3\} \subset \tau$ with $f(A_1) + f(A_2) + f(A_3) < i$. And $|W_i| \le k + \sum_{j=k-1}^{i-2} 3^j$.
\end{enumerate}
So we want to define $\mathcal X_i$, $N_i$, $W_{i+1}$ and a function $\mathcal X_i \rightarrow N_i$, $Y \mapsto v_Y$ satisfying the properties (i) - (iv) for $i+1$ instead of $i$.

For every $X = \{X_1,X_2,X_3\} \in \mathcal X_{i-1}$ let
\begin{equation}
\label{defcalX}
\mathcal X_{i,X} := \left\{\{X_1 \cup \{v_X\},X_2,X_3\}, \{X_1,X_2 \cup \{v_X\},X_3\}, \{X_1,X_2,X_3 \cup \{v_X\}\}\right\}.
\end{equation}
Then we can define
\begin{equation}
\label{N0}
\mathcal X_i := \bigcup_{X = \{X_1,X_2,X_3\} \in \mathcal X_{i-1}} \mathcal X_{i,X}.
\end{equation}
Hence, by induction hypothesis (ii), $|\mathcal X_i| \le 3|\mathcal X_{i-1}| \le 3^{i-1}$.

Next, let $\{A_1, A_2, A_3\} \subset \tau$ be an oriented separation triple that is not witnessed by $W_i$.
Then $\{A_1, A_2, A_3\} \subset \tau$ cannot be witnessed by $W_{i-1}$ either.
So according to induction hypothesis (i), there is $X = \{X_1,X_2,X_3\} \in \mathcal X_{i-1}$ such that \eqref{septripleX} holds.
As $\{A_1, A_2, A_3\}$ is not witnessed by $v_X \in W_{i}$, \eqref{defcalX} implies
\begin{equation}
\label{septripleXi}
A_j \cap Y_j = \emptyset \text{ for } j = 1,2,3
\end{equation}
for at least one $Y = \{Y_1,Y_2,Y_3\} \in \mathcal X_{i,X} \subset \mathcal X_i$. This shows (i).

For $Y \in \mathcal X_i$, we define $B_{Y,j}$, $j = 1,2,3$ (if existent) and $v_Y$ exactly as $B_{X,j}$ for $X \in \mathcal X_k$ above. Again $v_Y$ witnesses all triples $\{A_1, A_2, A_3\} \subset \tau$ satisfiying \eqref{septripleXi} where $f(A_1) + f(A_2) + f(A_3)$ is minimal. This shows (ii).

In order to show (iii), fix any triple $\{A_1, A_2, A_3\} \subset \tau$ satisfying \eqref{septripleXi} for some $Y \in \mathcal X_i$ and suppose for a contradiction that $f(A_1) + f(A_2) + f(A_3) < i$.  Without loss of generality $Y = \{X_1 \cup \{v_X\},X_2,X_3\}$ for some $X = \{X_1,X_2,X_3\} \in \mathcal X_{i-1}$. Thus, \eqref{septripleXi} implies \eqref{septripleX}. So by induction hypothesis (iii), we obtain that $f(A_1) + f(A_2) + f(A_3) = i-1$ is minimal among all triples satisfying \eqref{septripleX}. But then, by induction hypothesis (i), $v_X \in A_1 \cap A_2 \cap A_3$ which is a contradiction to \eqref{septripleXi} since $Y = \{X_1 \cup \{v_X\},X_2,X_3\}$. So (iii) holds.

Hence, every triple $\{A_1, A_2, A_3\} \subset \tau$ with $f(A_1) + f(A_2) + f(A_3) \le i$ is witnessed by either $W_{i}$ or
\begin{equation}
\label{Ni}
N_i := \{v_Y : Y \in \mathcal X_{i}\}.
\end{equation}
So we can define
\begin{equation}
\label{Wi}
W_{i+1}:= W_i \cup N_i.
\end{equation}
Then $W_{i+1}$ is as desired and by induction hypothesis (iv), we get
\begin{equation}
|W_{i+1}| \le |W_i| + |\mathcal X_i| \le k + \sum_{j=k}^{i-1} 3^j = \tfrac 1 2 (3^i-3^k) + k.
\end{equation}
which finally shows (iv) and finishes the proof.
\end{proof}

\subsection{Second upper bound}
\label{uppb2}

In this section, we prove a slightly better exponential upper bound for the size of a minimal witnessing set than in the previous section.
First, we will show a structural result that characterizes the order of a tangle.
From this, we will be able to conclude the upper bound.

First, we need another definition.
Fix $k \ge 1$. Then, for a tangle $\tau$, we call a series of minimal oriented separations $A_1,\dots, A_k \in \tau$ an \emph{intersection chain of length $k$ in $\tau$} if
\begin{equation}
\label{intchcond}
A_1^c \supsetneq A_1^c \cap A_2^c \supsetneq \dots
\supsetneq \bigcap_{i=1}^k A_i^c.
\end{equation}
Further, we refer to $I \subset V$ as a \emph{$k$-intersection} if $A_1,\dots,A_k$ is an intersection chain of maximum length such that $I = \bigcap_{i=1}^k A_i^c$.

\begin{prop}
\label{intchlord}
Let $S$ be a set of separations of any set $V$ and let $\tau$ be a $k$-tangle of $S$ regarding to some submodular order function $f$. Then the maximum length of an intersection chain in $\tau$ is $k$.
\end{prop}
\begin{proof}
Suppose for a contradiction, $\tau$ has an intersection chain $A_1, \dots, A_{k+1}$ of length $k+1$.
We will show by induction on $i = 1,...,k+1$ that $f\left(\bigcup_{j=1}^i A_j\right) \le k - i$.
For $i = k+1$, this will be the desired contradiction as order functions are defined to be non-negative.

Since $\tau$ is a $k$-tangle, we have $f(A_1) \le k-1$. So the induction starts.

Now assume that the induction hypothesis holds for some $i \in \{1,\dots,k\}$, i.~e. $f\left(\bigcup_{j=1}^i A_j\right) \le k - i$.

First, we show that $f\left(A_{i+1} \cap \bigcup_{j=1}^i A_j\right) \ge k$:
The intersection chain condition \eqref{intchcond} implies
\begin{equation*}
A_{i+1} \setminus \bigcup_{j=1}^i A_j
=\bigcup_{j=1}^{i+1} A_j \setminus \bigcup_{j=1}^i A_j
= \bigcap_{j=1}^i A_j^c\setminus \bigcap_{j=1}^{i+1} A_{j}^c \ne \emptyset.
\end{equation*}
Hence
\begin{equation*}
A_{i+1} \cap \bigcup_{j=1}^i A_j = A_{i+1} \setminus \left(A_{i+1} \setminus \bigcup_{j=1}^i A_j\right) \subsetneq A_{i+1}.
\end{equation*}

So $f \left(A_{i+1} \cap \bigcup_{j=1}^i A_j\right) < k$ would imply $A_{i+1} \cap \bigcup_{j=1}^i A_j \in \tau$ and, thus, contradict the minimality of $A_{i+1}$. Therefore $f\left(A_{i+1} \cap \bigcup_{j=1}^i A_j\right) \ge k$.

Now we are ready to show $f\left(\bigcup_{j=1}^{i+1} A_j\right) \le k - (i+1)$:
The induction hypothesis gives us $f\left(\bigcup_{j=1}^i A_j\right) \le k - i$.
Further, we have just shown $f \left(A_{i+1} \cap \bigcup_{j=1}^i A_j\right) \ge k$.
And $f(A_{i+1}) \le k-1$ since $A_{i+1} \in \tau$.
So, by submodularity, we obtain
\begin{align*}
f\left(\bigcup_{j=1}^{i+1} A_j\right) &=f\left(A_{i+1} \cup \bigcup_{j=1}^i A_j\right) \\
&\le f\left(\bigcup_{j=1}^i A_j\right) + f(A_{i+1}) - f \left(A_{i+1} \cap \bigcup_{j=1}^i A_j \right) \\
&\le k-i + k-1 - k = k - (i+1).
\end{align*}

So the induction hypothesis is also true for $i+1$ which completes the proof.
\end{proof}
Now using \cref{intchlord}, we can prove another upper bound for the size of a minimal witnessing set.

\begin{theorem}
\label{intch2}
Let $S$ be a set of separations of any set $V$ and let $\tau$ be a $k$-tangle of $S$ regarding to some submodular order function $f$. Then $\tau$ has a witnessing set of size at most $\frac {3^k - 1} 2$.
\end{theorem}
\begin{proof}

Let $W$ be a minimal witnessing set of the $k$-tangle $\tau$.
We will show by induction on $l = 1, \dots, k-1$ that every $(k-l)$-intersection $I$ satisfies
\begin{equation}
{|I \cap W| \le \frac {3^l-1} 2}.
\end{equation}
This implies the assertion of the theorem: Note that $A^c$ is a $1$-intersection for every minimal oriented separation $A \in \tau$.
Therefore, for any triple of minimal oriented separations ${\{A_1, A_2, A_3\} \subset \tau}$
\begin{equation}
\label{minsepW}
{|(A_1^c \cup A_2^c \cup A_3^c) \cap W|
\le 3 \cdot \frac {3^{k-1}-1} 2 = \frac {3^k - 3} 2}.
\end{equation}
And since $W$ is supposed to be minimal, for any point $w \in W$, there has to be a triple of minimal oriented separations $\{A_1, A_2, A_3\} \in \tau$ such that ${(A_1 \cap A_2 \cap A_3) \cap W = \{w\}}$: Otherwise $W \setminus \{w\}$ would still be a witnessing set of $\tau$ contradicting the minimality of $W$.
Hence, by \eqref{minsepW},
\begin{equation*}
|W| = |((A_1^c \cup A_2^c \cup A_3^c) \cap W) \overset \cdot \cup (A_1 \cap A_2 \cap A_3 \cap W)| \le \frac {3^k - 1} 2
\end{equation*}
as desired.

To start the induction, we first observe that $I^k \cap W = \emptyset$ for every $k$-intersection $I^k$:
Suppose for a contradiction, there would exist $w \in I^k \cap W$.
Let $A_1, \dots, A_k$ be an intersection chain of length $k$ that makes $I^k$ a $k$-intersection.
Assume $w \in A_{k+1}$ for some $A_{k+1} \in \tau$.
Without loss of generality, we can assume that $A_{k+1}$ is minimal in $\tau$.
But then $A_1, \dots, A_{k+1}$ would be an intersection chain of length $k+1$ -- contradicting \cref{intchlord}.
Hence, $w \notin A$ for all minimal oriented separations $A \in \tau$.
However, then $W \setminus \{w\}$ would still be a witnessing set of $\tau$ -- a contradiction to the minimality of $W$.
Thus, $I^k \cap W = \emptyset$ for every $k$-intersection $I^k$.

This implies that every $(k-1)$-intersection $I^{k-1}$ satisfies $|I^{k-1} \cap W| \le 1$:
Again, suppose for a contradiction, there would be $v,w \in I^{k-1} \cap W$ with $v \ne w$.
Suppose further that there is some $A_k \in \tau$ with $v \in A_k$ and $w \in A_k^c$ (or vice versa).
Once again, we can assume without loss of generality that $A_k$ is minimal in $\tau$.
Let $A_1, \dots, A_{k-1}$ be an intersection chain of length $k-1$ that makes $I^{k-1}$ a $(k-1)$-intersection.
Then $A_1, \dots, A_k$ is an intersection chain of length $k$ and $\bigcap_{i=1}^k A_i$ is a non-empty $k$-intersection since it contains either $v$ or $w$ -- a contradiction.
So there cannot be any $A_k \in \tau$ with $v \in A_k$ and $w \in A_k^c$ (or vice versa).
But then either $v$ or $w$ could be removed from $W$ and it would still be witnessing -- contradicting the minimality of $W$.
Therefore, $|I^{k-1} \cap W| \le 1 = \frac {3^1-1} 2$ for every $(k-1)$-intersection $I^{k-1}$. So the induction starts.

For the induction step, suppose the induction hypothesis has already been shown for some $l \in \{1, \dots, k-2\}$.
Let $I$ be a $(k-(l+1))$-intersection.
And let $w \in I \cap W$.
As $W$ is a minimal witnessing set, we find a triple $\{A_1, A_2, A_3\} \subset \tau$ such that $A_1 \cap A_2 \cap A_3 \cap W = \{w\}$ because otherwise $W \setminus \{w\}$ would still be a witnessing set of $\tau$.
Hence
\begin{equation*}
I \cap W = \{w\} \cup \bigcup_{i=1}^3 (I \cap A_i^c \cap W).
\end{equation*}
And since $I \cap A_i$ for $i = 1,2,3$ are $(k-l)$-intersections, we have $|I \cap W| \le 3 \cdot \frac{3^l-1} 2 + 1 = \frac {3^{l+1} - 1} 2$ which completes the proof.

\end{proof}

\begin{remark}
\label{intchgen}
\cref{intch2} can be applied to graph tangles even though graph tangles are witnessed by subgraphs instead of point sets. 
The best upper bound for the size of a minimal witnessing subgraph of a graph tangle of order $k$ is of size $O(3^{k^{k^5}})$ \cite{albrechtsen2024}.
We can improve this bound by using that every graph tangle induces an edge tangle which is a tangle of bipartitions. 
Thus, according to \cref{intch2}, every edge tangle of order $k$ has a witnessing set of size at most $\frac {3^k - 1} 2$.
Hence, every graph tangle of order $k$ is witnessed by a subgraph containing all the edges that witness its induced edge tangle.
This witnessing subgraph is of size at most $3^k-1$.
\end{remark}

\subsection{Lower bounds}
\label{lowb}
Suppose we are given two functions $\alpha, \beta: \mathbb N \rightarrow \mathbb N$ such that, for every $k \ge 1$, $\alpha(k)$ and $\beta(k)$ are upper bounds for the size of a minimal witnessing set, respectively, a cover of a $k$-tangle $\tau$ of $S$.
In the previous two sections, we have found such functions.
In this section, we provide two examples that bound those functions $\alpha$ and $\beta$ from below.
First, we show that the upper bound for the size of a minimal cover from Grohe and Schweitzer's \cref{coversize} \cite{grohe2015} is sharp.

\begin{example}
\label{examplecover}
For every $k \ge 1$, we can construct a $k$-tangle $\tau$ such that every cover of $\tau$ is of size at least $k$:
Fix $k \ge 1$.
Let $V$ be any set with $|V| \ge 3k-2$ and let $S$ be the set of all separations of $V$.
Consider the order function $f$ given by $f(\{A,A^c\}) = \min\{|A|,|A^c|\}$ for separations $\{A,A^c\} \in S$ which is submodular:
Let $A,B \subset V$.
Then
\begin{equation*}
f(A \cap B) + f(A \cup B) = \min\{|A \cap B|,|A^c \cup B^c|\} + \min\{|A \cup B|,|A^c \cap B^c|\}.
\end{equation*}
First, we consider the case $|A \cup B| \le |A^c \cap B^c|$. This implies $|A| \le |A^c|$ and $|B| \le |B^c|$. So we can conclude
\begin{align*}
f(A \cap B) + f(A \cup B) \le |A \cap B| + |A \cup B| = |A| + |B| = f(A) + f(B).
\end{align*}
For the case $|A^c \cup B^c| \le |A \cap B|$, we can argue analogously.
So the only remaining case is $|A \cap B| \le |A^c \cup B^c|$ and $|A^c \cap B^c| \le |A \cup B|$.
Here, we have
\begin{equation}
\label{minsubmod1}
f(A \cap B) + f(A \cup B) \le |A \cap B| + |A^c \cap B^c| \le |A \cap B| + |A \cup B| = |A| + |B|
\end{equation}
and analogously
\begin{equation}
\label{minsubmod2}
f(A \cap B) + f(A \cup B) \le |A \cap B| + |A^c \cap B^c| \le |A^c \cup B^c| + |A^c \cap B^c| = |A^c| + |B^c|.
\end{equation}
And further, we have
\begin{equation}
\label{minsubmod3}
f(A \cap B) + f(A \cup B) \le |A \cap B| + |A^c \cap B^c| \le |A| + |B^c|
\end{equation}
and
\begin{equation}
\label{minsubmod4}
f(A \cap B) + f(A \cup B) \le |A \cap B| + |A^c \cap B^c| \le |A^c| + |B|.
\end{equation}
Together, the equations \eqref{minsubmod1}, \eqref{minsubmod2}, \eqref{minsubmod3} and \eqref{minsubmod4} imply
\begin{align*}
f(A \cap B) + f(A \cup B) \le \min\{|A|,|A^c|\} + \min\{|B|,|B^c|\} = f(A) + f(B).
\end{align*}
This finally shows the submodularity of $f$.
So $\tau = \{A \subset V: |A^c| < k\}$ is a $k$-tangle of $S$.
But $\tau$ cannot be covered by a set $U \subset V$ with $|U| < k$.
\end{example}

The next example yields a lower bound for a function $\alpha$ in $k$ bounding the size of a minimal witnessing set of a $k$-tangle from above.

\begin{example}
\label{hypwbsp}
We consider the example that Diestel, Elbracht and Jacobs \cite{diestel2021} brought up to show that there are tangles without a guiding set:
Fix $k \ge 4$.
Let $V = \{T \subset \{1, \dots, k\}: |T| = 3\}$.
Then $|V| = \binom k 3.$
For $j \in \{1, \dots, k\}$, let $A_j := \{T \in V: j \in T\}$.
Then $\tau = \{A_j: j \in \{1,\dots,k\}\}$ is a tangle of $S = \{\{A_j,A_j^c\}: j \in \{1,\dots,k\}\}$:
Let $\{j_1,j_2,j_3\} \subset \{1,\dots,k\}$ with $|\{j_1,j_2,j_3\}|=3$.
Then ${\{j_1,j_2,j_3\}} \in A_{j_1} \cap A_{j_2} \cap A_{j_3}$ and, hence, the consistency condition \eqref{conscond} is satisfied.

As every tangle, $\tau$ is witnessed by its groundset $V$.
But let $W \subsetneq V$ be any proper subset of $V$.
Then we can pick $T = {\{j_1,j_2,j_3\}} \in V \setminus W$.
And since $A_{j_1} \cap A_{j_2} \cap A_{j_3} = \{T\}$, $W$ is not witnessing $\tau$.
Thus, the smallest witnessing set of $\tau$ has the size $|V| = \binom k 3$.

The following proposition shows that we can extend $\tau$ to a $k$-tangle $\tau^*$ on the set $S^*$ of all separations of $V$ regarding to some order function $f^*$ on $S^*$ such that $W$ still is a minimal witnessing set of $\tau^*$.

\end{example}

\begin{lemma}
\label{nonfakeprop}
Let $S$ be a set of separations of any set $V$ and let $\tau$ be a tangle of $S$ with $k$ minimal oriented separations.
Then there is a submodular order function $f^*$ on the set of all separations $S^*$ of $V$ such that $\tau$ extends to a $k$-tangle $\tau^*$ of $S^*$ with respect to $f^*$ such that the minimal separations of $\tau^*$ are the minimal separations of $\tau$.
\end{lemma}


\begin{proof}
Let $A_1, \dots, A_k \in \tau$ be the $k$ minimal oriented separations in $\tau$.
For every $v \in V$ let
\begin{equation}
T_v := \{j \in \{1, \dots, k\}: v \in A_j\}.
\end{equation}
Now we can endow the set $S^*$ of all separations of $V$ with an order function $f^*:S^* \rightarrow \mathbb N_0$ such that every separation $(A,A^c) \in S^*$ with $A_i \subset A$ or $A_i \subset A^c$ for at least one $i \in \{1, \dots, k\}$ has order $< k$.
(For these separations, the consistency condition \eqref{conscond} implies that $(A,A^c)$ will be oriented towards $A$ if $A_i \subset A$ for some $i \in \{1, \dots, k\}$ and towards $A^c$ otherwise.)
For every separation $s = \{A,A^c\} \in S^*$, let $U_s \subset \{1,\dots,k\}$ be a set of maximum cardinality such that $\{v \in V: T_v \cap U_s \ne \emptyset\} \subset A$ or $\{v \in V: T_v \cap U_s \ne \emptyset\} \subset A^c$.
Then $|U_s|$ counts the maximum number of minimal oriented separations $A_i$ of $\tau$ that are all contained in the same side of $s$.
We set
\begin{equation}
f^*(s) := k - |U_s|.
\end{equation}
To show that $f^*$ is submodular, let $r = \{A,A^c\}, s = \{B,B^c\} \in S^*$. We have to prove the submodularity condition $f^*(A \cap B)+f^*(A \cup B) \le f^*(A) + f^*(B)$.
Without loss of generality, we can assume that one of the following two cases applies:

Case 1: $\{v \in V: T_v \cap U_r \ne \emptyset\} \subset A$ and $\{v \in V: T_v \cap U_s \ne \emptyset\} \subset B$. Then
\begin{equation}
\{v \in V: T_v \cap (U_r \cap U_s) \ne \emptyset\} \subset A \cap B
\end{equation}
and
\begin{equation}
\{v \in V: T_v \cap (U_r \cup U_s) \ne \emptyset\}
= \{v \in V: (T_v \cap U_r) \cup (T_v \cap U_s) \ne \emptyset\}
\subset A \cup B.
\end{equation}
Since $U_{\{A \cap B, A^c \cup B^c\}}$ and $U_{\{A \cup B, A^c \cap B^c\}}$ are chosen of maximum cardinality this implies $|U_{\{A \cap B, A^c \cup B^c\}}| \ge |U_r \cap U_s|$ and $|U_{\{A \cup B, A^c \cap B^c\}}| \ge |U_r \cup U_s|$.
Hence
\begin{align*}
f^*(A \cap B)+f^*(A \cup B) &= k - |U_{\{A \cap B, A^c \cup B^c\}}| + k - |U_{\{A \cup B, A^c \cap B^c\}}| \\
&\le k - |U_r \cap U_s| + k - |U_r \cup U_s| \\
&= k - |U_r| + k - |U_s| \\
&= f^*(A) + f^*(B).
\end{align*}

Case 2: $\{v \in V: T_v \cap U_r \ne \emptyset\} \subset A$ and $\{v \in V: T_v \cap U_s \ne \emptyset\} \subset B^c$. Then
\begin{equation}
\{v \in V: T_v \cap U_s \ne \emptyset\} \subset B^c \subset A^c \cup B^c
\end{equation}
and
\begin{equation}
\{v \in V: T_v \cap U_r \ne \emptyset\} \subset A \subset A \cup B.
\end{equation}
Thus, $|U_{\{A \cap B, A^c \cup B^c\}}| \ge |U_s|$ and $|U_{\{A \cup B, A^c \cap B^c\}}| \ge |U_r|$.
Hence
\begin{align*}
f^*(A \cap B)+f^*(A \cup B) &= k - |U_{\{A \cap B, A^c \cup B^c\}}| + k - |U_{\{A \cup B, A^c \cap B^c\}}| \\
&\le k - |U_r| + k - |U_s| \\
&= f^*(A) + f^*(B).
\end{align*}
So $f^*$ is a submodular order function. Let
\begin{equation}
\tau^* = \{A \in \ora {S^*}: U_{\{A,A^c\}} \ne \emptyset, \{v \in V: T_v \cap U_{\{A,A^c\}} \ne \emptyset\} \subset A\}.
\end{equation}
To see that $\tau^*$ extends $\tau$, first, note that for no separation $s=\{A,A^c\} \in S^*$ we find two non-empty sets $U_1, U_2 \ne \emptyset$ such that $\{v \in V: T_v \cap U_1 \ne \emptyset\} \subset A$ and $\{v \in V: T_v \cap U_2 \ne \emptyset\} \subset A^c$:
Otherwise, there would be $i \in U_1$ and $j \in U_2$ which would imply $A_i \subset \{v \in V: T_v \cap U_1 \ne \emptyset\} \subset A$ and $A_j \subset \{v \in V: T_v \cap U_2 \ne \emptyset\} \subset A^c$ and, hence, $A_i \cap A_j = \emptyset$ -- a contradiction to the consistency condition \eqref{conscond} for the tangle $\tau$.
Now let $A \in \tau$.
Then we find $j \in \{1, \dots, k\}$ with $A_j \subset A$ (since $A_1,\dots,A_k$ are the maximal oriented separations of $\tau$).
But since $A_j = \{v \in V: T_v \cap \{j\} \ne \emptyset\}$ this implies $j \in U_{\{A,A^c\}} \ne \emptyset$ and $\{v \in V: T_v \cap U_{\{A,A^c\}} \ne \emptyset\} \subset A$ and, hence, $A \in \tau^*$.

Next, we show that $\tau^*$ is a tangle: Let $B_1, B_2, B_3 \in \tau^*$.
As $U_{\{B_j,B_j^c\}} \ne \emptyset$, for $j = 1,2,3$, we find $i_1,i_2,i_3 \in \{1, \dots, k\}$ with $i_j \in U_{\{B_j,B_j^c\}}$ and, thus, $A_{i_j} \subset B_j$ for $j = 1,2,3$.
But since $\tau$ is a tangle there is $v \in A_{i_1} \cap A_{i_2} \cap A_{i_3} \subset B_1 \cap B_2 \cap B_3$. So the consistency condition \eqref{conscond} is satisfied and $\tau^*$ is a tangle.

Finally, note that $\tau^*$ is a $k$-tangle with respect to $f^*$ because the condition $U_{\{A,A^c\}} \ne \emptyset$ is equivalent to $f^*(A) < k$.
This completes the proof.
\end{proof}

\section{Guiding functions and guiding sets of tangles}
\label{guidsets}

In this section, we address the question of when a tangle has a guiding set.
For this, we first demonstrate how we can determine the maximum reliability of a guiding function of a tangle.
Subsequently, we show that under certain circumstances, especially if the guiding function of a tangle has sufficiently high reliability, we can conclude the existence of a guiding set of this tangle.

\subsection{Guiding functions of tangles}
\label{guidfct}

First, we will prove a generalization of a duality theorem by Diestel, Elbracht and Jacobs \cite[Thm.~8]{diestel2021}.
It will guarantee that we either find a guiding function of a tangle with some given reliability or a function on the set of separations $S$ witnessing that there cannot be such a guiding function.
We will derive this theorem from the strong duality theorem of linear optimization (see e.g. ~\cite[ch.\ 6.2]{fang1993}).

\begin{theorem}[Strong duality theorem of linear optimization]
\label{strongdual}
Let $n,m \in \mathbb N$, $A \in M(n \times m, \mathbb R)$, $b \in \mathbb R^m$ and $c \in \mathbb R^n$.
Then the following holds:
The primal linear optimization problem
\begin{equation}
\label{primal}
c^* := \max_{x \in \mathbb R_+^n: Ax \le b} c^Tx
\end{equation}
has a finite optimal solution $c^*$ if and only if its dual
\begin{equation}
\label{dual}
b^* := \min_{y \in \mathbb R_+^m: A^T y \ge c} b^Ty
\end{equation}
has a finite optimal solution $b^*$. In either case $c^* = b^*$.
\end{theorem}

\begin{theorem}
\label{alpha_beta}
Let $S$ be a set of separations of any set $V$ and let $\tau$ be a $k$-tangle of $S$ regarding to some submodular order function $f$. Further, let $\rho \in [0,1]$.
Then there is precisely one of the following two functions: Either there is a function $h: S \rightarrow \mathbb R_+$ with $\sum_{s \in S} h(s) = 1$ such that for all $v \in V$
\begin{equation}
\label{bed_alpha}
\sum_{s \in S: \text{ } v \in \tau(s)} h(s) < \rho
\end{equation}
or there is a guiding function $g: V \rightarrow \mathbb R_+$ of $\tau$ with reliability $\rho$, i.~e. $\sum_{v \in V} g(v) = 1$ and for all $s \in S$
\begin{equation}
\label{bed_beta}
\sum_{v \in \tau(s)} g(v) \ge \rho.
\end{equation}
\end{theorem}

\begin{proof}
Let $n = |S|$, $m = |V|$, $c = (\rho,\dots,\rho)^T \in \mathbb R_+^n$, $b = (\rho,\dots,\rho)^T \in \mathbb R_+^m$ and $A \in M(n \times m, \mathbb R)$ with entries $a_{vs} = \mathbb 1_{\{v \in \tau(s)\}}$ for $v \in V$ and $s \in S$.
Consider the linear optimization problem
\begin{equation}
\label{primal}
c^{*'}_\rho := \max_{x \in \mathbb R_+^n: Ax \le b} c^Tx = \max_{h \in M_{\rho}'}\sum_{s \in S} h(s)
\end{equation}
for
\begin{equation}
\label{M'}
M_{\rho}' := \left\{h \in \mathbb R_+^S : \sum_{s \in S: \text{} v \in \tau(s)} h(s) \le \rho, \forall v \in V\right\}.
\end{equation}
The dual of \eqref{primal} is
\begin{equation}
\label{dual}
b^* := \min_{y \in \mathbb R_+^m: A^T y \ge c} b^Ty = \min_{g \in N_{\rho}}\sum_{v \in V} g(v)
\end{equation}
for 
\begin{equation}
\label{N}
N_{\rho} := \left\{g \in \mathbb R_+^V : \sum_{v \in V: \text{} v \in \tau(s)} g(s) \ge \rho, \forall s \in S\right\}.
\end{equation}
Both linear optimization problems have a feasible solution given by the constant functions $\rho^S \in M'_\rho$ and $\rho^V \in N'_\rho$.
Further ${c^*_\rho}'$ is bounded from above by ${c^*_\rho}' \le \rho |S|$ and $b^*$ is bounded from below by $b^* \ge \rho |V|$.
So the simplex algorithm finds an optimal solution for each of them \cite[ch. 1]{fang1993}.
Hence, by \cref{strongdual}, ${c^*_\rho}' = b^*_\rho$.
Thus, there are the following two cases:

Case 1: ${c^*_\rho}' = b^*_\rho > 1$.
Then for all functions $g \in \mathbb R_+^V$ satisfying \eqref{bed_beta}, ${\sum_{v \in V}g(v) > 1}$.
So there is no guiding function $g$ with reliability $\rho$.
But
\begin{equation*}
h^*:= \frac 1 {c_\rho^*} \argmax_{h \in M_\rho'}\sum_{s \in S} h(s)
\end{equation*}
satisfies $\sum_{s \in S} h^*(s) = 1$ and condition $\eqref{bed_alpha}$ (for $h^*$ instead of $h$).

Case 2: ${c^*_\rho}' = b^*_\rho \le 1$.
Then
\begin{equation}
g^*:= \frac 1 {b_\rho^*} \argmin_{g \in N_\rho}\sum_{v \in V} g(v)
\end{equation}
is a guiding function with reliability $\rho$.
But for all $h \in M_\rho'$, $\sum_{s \in S} h(s) \le 1$.
Hence, if $h \in M_\rho'$ and $\sum_{s \in S} h(s) = 1$ then there is at least one point $v \in V$ with $\sum_{s \in S: \text{} v \in \tau(s)} h(s) = \rho$.
Therefore, for all $h \in M_\rho$, we have $\sum_{s \in S} h(s) < 1$.
This finally implies that there is no $h \in \mathbb R_+^S$ satisfying $\sum_{s \in S} h(s) = 1$ and \eqref{bed_alpha} which completes the proof.
\end{proof}

Note that a tangle $\tau$ has a guiding function with reliability $\rho$ if and only if
\begin{equation}
\label{opt_probl}
c^*_\rho := \min_{g \in N_\rho} \sum_{v \in V} g(v) \le 1
\end{equation}
for $N_\rho$ as defined in \eqref{N}.
Notice further that the maximum reliability of a guiding function of $\tau$ is that $\rho$ for which $c^*_\rho = 1$.
Therefore, in order to find a guiding function of $\tau$ with maximum reliability, we can solve the linear optimization problem \eqref{opt_probl} for any $\rho \in (0,1]$.
Then, as this optimization problem is linear, a guiding function with maximum reliability is given by
\begin{equation}
g := \frac 1 {c^*_\rho}\argmin_{g \in N_\rho} \sum_{v \in V} g(v).
\end{equation}

Linear optimization problems can be solved in polynomial time for example by using an interior point method (see e.g.~\cite{karmarkar1984}).
Thus, if existing, a guiding function with maximum reliability can always be computed in polynomial time in the size of $S$.
However, this does not have to be the case for a guiding set.
But we can, under certain circumstances, conclude the existence of a guiding set from the existence of a guiding function, by using a probabilistic argument (see \cref{existguid}).

\subsection{Guiding sets of tangles}
\label{existguid}

In this section, we use a probabilistic argument to show that a tangle $\tau$ has a guiding set under certain assumptions on $\tau$.
For this, we fix a guiding function $g$ of a tangle $\tau$ and consider a random subset $G \subset V$ that contains the point $v \in V$ with probability $g(v)$.
We show that under certain assumptions on $g$ this implies that, with probability greater than zero, $G$ is a guiding set of $V$.
So it is possible to choose $G \subset V$ such that $G$ guides V.

Recall that a random variable $X$ on some probability space $(\Omega, \mathcal A, \mathbb P)$ is \emph{Bernoulli distributed} with parameter $p$ if $\mathbb P (X = 1) = p$ and $\mathbb P (X = 0) = 1-p$. We denote this by $X \sim \text{Ber}(p)$.
The expected value of a Bernoulli-distributed random variable $X$ is given by $\mathbb E[X] = p$ and its variance calculates to $\text{Var}(X) = p(1-p)$.

\begin{theorem}
\label{guidfctsetthm}
Let $S$ be a set of separations of any set $V$ and let $\tau$ be a tangle of $S$. Suppose that $\tau$ has a guiding function $g \in \mathbb R_+^V$ with reliability $\rho := \min_{A \in \tau} \rho_{A}\in (\frac 1 2, 1]$
where $\rho_{A} := \sum_{v \in A} g(v)$, $A \in \tau$. Let $\eta := \frac 1 {\max_{v \in V}g(v)}$ and suppose further that
\begin{equation}
\label{assumption_prob}
\frac {\sum_{v \in V} g(v) (1-g(v) \eta)}{4\eta}\cdot\sum_{A \in \tau_{\min}} \frac{1}{(\rho_{A}  - \tfrac 1 2)^2} < 1
\end{equation}
where $\tau_{\min} \subset \tau$ denotes the set of all minimal oriented separations in $\tau$. 
Then $\tau$ has a guiding set $G \subset V$.
\end{theorem}

\begin{proof}
Note that the definition of $\eta$ assures $g(v) \eta \in [0,1]$ for all $v \in V$.
Let
\begin{equation}
\label{defXv}
X_v \sim \text{Ber}(g(v) \eta), \quad v \in V
\end{equation}
be independent Bernoulli distributed random variables on some probability space $(\Omega, \mathcal A, \mathbb P)$. Then
\begin{equation}
\label{defD}
G = \{v: X_v = 1\}
\end{equation}
is a random subset of $V$.
Let
\begin{equation}
\label{defYs}
Y_{A} = \sum_{v \in G} \mathbb 1_{\{v \in A\}}, \quad A \in \tau_{\min}.
\end{equation}
If we can show that
\begin{equation}
\label{zzprob}
 \mathbb P(Y_{A} - \tfrac 1 2|G| > 0, \quad \forall A \in \tau_{\min}) > 0
\end{equation}
then there is some $\omega \in \Omega$ for which $G(\omega)$ is a guiding set of $\tau$.
Therefore, we fix $A \in \tau_{\min}$ and compute the expected value of $Y_{A} - \tfrac 1 2|G|$
\begin{align}
\label{Ew}
\mathbb E[Y_{A} - \tfrac 1 2|G|]
&\overset{\eqref{defYs}}{=} \mathbb E \left[\sum_{v \in G}\left(\mathbb 1_{\{v \in A\}} - \tfrac 1 2 \right) \right]\nonumber\\
&\overset{\eqref{defD}}{=} \sum_{v \in V}\mathbb P(X_v = 1) \left(\mathbb 1_{\{v \in A\}} - \tfrac 1 2 \right)\nonumber\\
&\overset{\eqref{defXv}}{=} \sum_{v \in V} g(v) \eta \left(\mathbb 1_{\{v \in A\}} - \tfrac 1 2\right)\nonumber\\
&= \left(\rho_{A}  - \tfrac 1 2\right)\eta > 0.
\end{align}
Next, we compute the variance of $Y_{A} - \tfrac 1 2|G|$. In the third step of the following calculation, we use that the random variables $X_v$ are independent and, hence, $\text{Var}(\sum_{v \in V} X_v) = \sum_{v \in V} \text{Var}(X_v)$. And in the fourth step, we use that for $a \in \mathbb R$ and a random variable $Z$, $\text{Var}(aZ) = a^2\text{Var}(Z)$.
\begin{align}
\label{Var}
\text{Var}(Y_{A} - \tfrac 1 2|G|)
&\overset{\eqref{defYs}}{=} \text{Var} \left(\sum_{v \in G}\left(\mathbb 1_{\{v \in A\}} - \tfrac 1 2\right) \right)\nonumber\\
&\overset{\eqref{defD}}{=} \text{Var} \left(\sum_{v \in V}X_v \left(\mathbb 1_{\{v \in A\}} - \tfrac 1 2\right) \right)\nonumber\\
&= \sum_{v \in V} \text{Var} \left(X_v\left(\mathbb 1_{\{v \in A\}} - \tfrac 1 2\right)\right)\nonumber\\
&= \tfrac 1 4 \sum_{v \in V} \text{Var} (X_v)\nonumber\\
&\overset{\eqref{defXv}}{=} \tfrac 1 4 \sum_{v \in V} g(v) \eta (1-g(v) \eta).
\end{align}
Now we can use Chebychev's inequality to obtain
\begin{align*}
\mathbb P \left(Y_{A} - \tfrac 1 2|G| \le 0 \right)
&\le \mathbb P\left(\left|Y_{A} - \tfrac 1 2|G| - \mathbb E[Y_{A} - \tfrac 1 2|G|]\right| \ge \mathbb E[Y_{A} - \tfrac 1 2|G|]\right)\\
&\le \frac{\text{Var}(Y_{A} - \tfrac 1 2|G|)}{\mathbb E[Y_{A} - \tfrac 1 2|G|]^2}\\
&\le \frac{\tfrac 1 4 \sum_{v \in V} g(v) \eta(1-g(v) \eta)}{((\rho_{A}  - \tfrac 1 2) \eta)^2}\\
&= \frac{\sum_{v \in V} g(v) (1-g(v) \eta)}{4(\rho_{A}  - \tfrac 1 2)^2 \eta}.
\end{align*}
Finally, we can show \eqref{zzprob}:
\begin{align*}
\mathbb P(Y_{A} - \tfrac 1 2|G| > 0, \quad \forall A \in \tau_{\min})
&\ge 1 - \sum_{A \in \tau_{min}} \mathbb P(Y_{A} - \tfrac 1 2|G| \le 0)\\
&= 1 - \frac {\sum_{v \in V} g(v) (1-g(v) \eta)}{4\eta}\cdot\sum_{A \in \tau_{\min}} \frac{1}{(\rho_{A}  - \tfrac 1 2)^2} \\
&\overset{\eqref{assumption_prob}}{>} 0.
\end{align*}
\end{proof}

Condition \eqref{assumption_prob} clearly holds if $g$ already is a guiding set, i.e. if $g(v) = \frac 1 {|G|} \mathbb 1_{\{v \in G\}}$ for some $G \subset V$:
Then $\eta = \frac 1 {\max_{v \in V}g(v)} = |G|$ and, thus, the left side of \eqref{assumption_prob} is zero.
Non-trivial cases where \eqref{assumption_prob} holds might especially occur when $g$ has high reliability and when $|V|$ is much bigger than the number of minimal oriented separations in $\tau$.
The following corollary yields a simplified version of \cref{guidfctsetthm}.

\begin{cor}
Let $S$ be a set of separations of any set $V$ of $n$ points and let $\tau$ be a tangle of $S$. Suppose that $\tau$ has a guiding function $g \in \mathbb R_+^V$ with reliability $\rho = \frac 1 2 + \delta$ for some $\delta \in (0,\frac 1 2]$.
Let $\epsilon := \frac 1 {n \max_{v \in V g(v)}}$ and suppose further that
\begin{equation}
|\tau_{\min}| < 4 \delta^2 \epsilon^2 |V|
\end{equation}
where $\tau_{\min} \subset \tau$ denotes the set of all minimal oriented separations in $\tau$. 
Then $\tau$ has a guiding set $G \subset V$.
\end{cor}
\qed
\section{Interplay of witnessing and guiding sets of tangles}
\label{interplay}

In this section, we deal with the combination of witnessing and guiding sets to tangle associated clusters by providing two counterexamples.
First, we give an example showing that the cluster candidates that Diestel~\cite[ch. 14.2]{diestel2024} suggested need not witness their tangle, i.e.\ that a set which is guiding a non-principal and non-fake tangle with maximum reliability does not necessarily witness this tangle.

\begin{example}
\label{fuzzyclusterce}
Let $n \ge 6$ with $3 | n$ and $k = \frac 2 3 n$. Let
\begin{equation}
\label{notbothwitn}
W:= \{T \subset \mathbb Z /n \mathbb Z = \{[0]_n,\dots,[n-1]_n\}: |T| = k-1\},
\end{equation}
\begin{equation}
\label{notbothguid}
G = \{\{[0]_n+j,\dots, [k]_n+j\}: j \in \mathbb Z/ n \mathbb Z\}
\end{equation}
and $V = W \cup G$.
And let $A_j := \{v \in V: j \in v\}$ for $j \in \mathbb Z/ n \mathbb Z$.
Then $\tau := \{A_j: j \in \mathbb Z/ n \mathbb Z\}$ is a tangle of $S := \{\{A_j,A_j^c\}: j \in \mathbb Z/ n \mathbb Z\}$, $W$ is a witnessing set of $\tau$ and $G$ is a guiding set of $\tau$.
By \cref{nonfakeprop}, $\tau$ extends to an $n$-tangle $\tau^*$ of the set $S^*$ of all separations of $V$ regarding to some order function $f^*$. And since the maximal oriented separations are preserved, $W$ is still a witnessing set of $\tau^*$ and $G$ a guiding set of $\tau^*$.
The reliability $\rho_G$ of $G$ is maximal because the reliability $\rho_H$ of any non-empty set $H \subset V$ satisfies
\begin{align*}
\rho_H = \min_{j \in \mathbb Z/ n \mathbb Z} \frac{|A_{[j]_n} \cap H|}{|H|}
&\le \frac 1 n \sum_ {j \in \mathbb Z/ n \mathbb Z} \frac{A_{[j]_n} \cap H}{|H|} \\
&= \frac 1 n \frac 1 {|H|} \sum_ {j \in \mathbb Z/ n \mathbb Z} \sum_{v \in H} \mathbb 1_{\{v \in A_{[j]_n}\}} \\
&= \frac 1 n \frac 1 {|H|} \sum_{v \in H} \sum_ {j \in \mathbb Z/ n \mathbb Z} \mathbb 1_{\{j \in v\}} \\
&\le \frac {|H|k}{n \cdot |H|} = \frac k n = \rho_G. 
\end{align*}
As $\tau^*$ is a $n$-tangle of the set $S^*$ of all separations of $V$, $\tau^*$ is not fake.
Since there is no vertex $v \in V$ witnessing all triples of oriented separations in $\ora {S^*}$, $\tau^*$ is non-principal.
So $G$ guides the non-principal and non-fake tangle $\tau^*$ with maximum reliability.
But $G$ is not witnessing $\tau^*$ as, for instance, the oriented separation triple $\{A_{[1]_n},A_{[\frac n 3 + 1]_n},A_{[\frac {2n}3 + 1]_n}\}$ is not witnessed by $G$.
\end{example}

This suggests that requiring the aforementioned cluster candidates to witness their tangle, too, can be a reasonable addition.
But this could lead to ruling out these cluster candidates: The next example shows that there are tangles that have a witnessing set (which is a proper subset of the groundset $V$) and a guiding set but no set that is both witnessing and guiding.

\begin{example}
\label{notboth}
Let $k \ge 5$ and $n = 2k - 1$ with $3|n$.
Then we can define 
\begin{equation}
W:= \{T \subset \mathbb Z /n \mathbb Z = \{[0]_n,\dots,[n-1]_n\}: |T| = 3\},
\end{equation}
Further we define $G$, $V$, $S$, $\tau$ and $\tau^*$ exactly as in \cref{fuzzyclusterce}.
So again, $\tau^*$ is a non-principal non-fake tangle of the set $S^*$ of all separations of $V$, $W$ is a witnessing set of $\tau^*$ and $G$ a guiding set of $\tau^*$.

Suppose there would be a set $U \subset V$ that is both witnessing and guiding $\tau^*$.
Consider the set
\begin{equation*}
\{\{A_{[0]_n}+j,A_{[\frac 1 3 n + 1]_n}+j,A_{[\frac 2 3 n + 2]_n}+j\}: j \in \mathbb Z/ n \mathbb Z\}\}
\end{equation*}
of oriented separation triples.
No separation triple in this set is witnessed by a point in $G$.
Hence, $U$ must contain at least $n$ points from $W$.
But then, for the reliability $\rho_U$ of $U$ we get
\begin{align*}
\rho_U = \min_{j \in \mathbb Z/ n \mathbb Z} \frac{|A_{[j]_n} \cap U|}{|U|}
&\le \frac 1 n \sum_ {j \in \mathbb Z/ n \mathbb Z} \frac{|A_{[j]_n} \cap U|}{|U|} \\
&= \frac 1 n \frac 1 {|U|} \sum_ {j \in \mathbb Z/ n \mathbb Z} \sum_{v \in U} \mathbb 1_{\{v \in A_{[j]_n}\}} \\
&= \frac 1 n \frac 1 {|U|} \sum_{v \in U} \sum_ {j \in \mathbb Z/ n \mathbb Z} \mathbb 1_{\{j \in v\}} \\
&\le \frac {3n +kn}{n \cdot 2n} 
= \frac {3 + \frac {n+1} 2}{2n} = \frac {n+7}{4n} \le \frac{16}{36} < \frac 1 2.
\end{align*}
Thus, $|U|$ cannot be a guiding set of $\tau^*$.
This shows that there is no set $U \subset V$ both witnessing and guiding $\tau^*$.
\end{example}

\section*{Acknowledgements}
I want to thank Sandra Albrechtsen for many helpful comments and suggestions.

\bibliography{References.bib}
\bibliographystyle{plain}

\end{document}